\RequirePackage{fix-cm}
\documentclass[smallextended]{svjour3}       
\smartqed  
\usepackage{graphicx}
\usepackage{color}
\usepackage{amsmath,amssymb,hyperref,mathrsfs,graphicx,mathtools,units}
\usepackage{amsfonts}
\usepackage{enumerate}

\definecolor{cadmiumgreen}{rgb}{0.0, 0.42, 0.24}

\newtheorem{thm}{Theorem}

\newtheorem{prop}{Proposition}[section]
\newtheorem{lem}[prop]{Lemma}

\newtheorem{rmk}[prop]{Remark}

\newcommand{\eps}{\varepsilon}

\newcommand{\R}{\mathbb{R}}

\newcommand{\N}{\mathbb{N}}
\newcommand{\Z}{\mathbb{Z}}

\newcommand{\dt}{\delta}

\newcommand{\bl}{\mathbf{l}}

\newcommand{\bi}{\mathbf{i}}
\newcommand{\bj}{\mathbf{j}}

\newcommand{\be}{\begin{equation}} 
\newcommand{\ee}{\end{equation}}
\newcommand{\bea}{\begin{eqnarray}} 
\newcommand{\eea}{\end{eqnarray}}
\newcommand{\bean}{\begin{eqnarray*}} 
\newcommand{\eean}{\end{eqnarray*}}
 
\def\dx{\,{\rm d}x}

\def\dt{\,{\rm d}t}

\journalname{Journal of Dynamics and Differential Equations}

\begin{document}

\titlerunning{Sharp Sobolev estimates for an aggregation--diffusion equation}
\title{Sharp Sobolev estimates for concentration of solutions to an aggregation--diffusion equation}

\dedication{In memory of Genevi{\`e}ve Raugel, colleague and mentor for many mathematicians in dynamics, PDEs and numerical analysis
}


\author{Piotr Biler \and
        Alexandre Boritchev \and
				Grzegorz Karch \and
				Philippe Lauren{\c c}ot  
}

\institute{P.~Biler \at
              Instytut Matematyczny, Uniwersytet Wroc{\l}awski, pl. Grunwaldzki 2/4, 50-384  Wroc{\l}aw, Poland, 
\email{piotr.biler@uwr.edu.pl}         
           \and
           A.~Boritchev \at Universit\'e Claude Bernard -- Lyon 1, CNRS UMR 5208, Institut Camille Jordan, F-69622 Villeurbanne, France,
\email{alexandre.boritchev@gmail.com}
					\and
           G.~Karch \at Instytut Matematyczny, Uniwersytet Wroc{\l}awski, pl. Grunwaldzki 2/4, 50-384 Wroc{\l}aw, Poland,
\email{grzegorz.karch@uwr.edu.pl}
					\and
           P.~Lauren{\c c}ot \at Institut de Math\'ematiques de Toulouse, UMR 5219, Universit\'e de Toulouse, CNRS, F-31062, Toulouse Cedex 9,
 France, \email{Philippe.Laurencot@math.univ-toulouse.fr}
}

\date{Received: date / Accepted: date}

\maketitle

\begin{abstract} 
We consider the  drift-diffusion equation
$$
u_t-\eps\Delta u+\nabla\cdot(u\ \nabla K\ast u)=0
$$
 in the whole space with global-in-time solutions bounded in all Sobolev spaces; for simplicity, we restrict ourselves to the model case $K(x)=-|x|$.
\\
We quantify the mass concentration phenomenon, a genuinely nonlinear effect, for radially symmetric solutions of  this equation  for small diffusivity $\eps$ studied in our previous paper \cite{BBKL20}, obtaining optimal sharp upper and lower bounds for Sobolev norms. 
\end{abstract}

\section{Introduction}

We study the nonlinear nonlocal equation 
\begin{equation} \label{ADE}
u_t-\eps\Delta u+\nabla\cdot(u\ \nabla K\ast u)=0, \qquad x\in \mathbb{R}^N, \;t>0, 
\end{equation}
where $\eps>0$ is the diffusivity. We consider the simplest case of a \textit{pointy potential}. In other words, to clarify the presentation we restrict ourselves to the radially symmetric kernel $K(x)=-|x|$ which has a mild singularity at the origin. Equation~\eqref{ADE} belongs to a class of models describing numerous phenomena from biology and astrophysics; see the review \cite{HP09} and \cite[Introduction]{B-DeG} for further references.
\\ \indent
We make the following assumptions on the initial condition $u(\cdot,0) \equiv u_0$:
\begin{enumerate}[(A)]
\item \label{A1}  The function $u_0$ is  $C^\infty$-smooth, bounded and integrable along with all its derivatives. In other words, 
$$
u(\cdot,0) \equiv u_0 \in \bigcap_{k \ge 0,\ 1 \le p \le \infty} W^{k,p}(\mathbb{R}^N),
$$
where $W^{k,p}(\mathbb{R}^N)$ are the usual Sobolev spaces (see Section~\ref{nota}).
\smallskip
\item \label{A2} The function $u_0$ is non-negative and radially symmetric.
\smallskip
\item \label{A3}  The mass of $u_0$ is sufficiently concentrated: 
\begin{equation} \label{u0conc}
\int_{\R^n}{|x| u_0(x) dx} < \infty.
\end{equation}
\end{enumerate}
\indent
Since \eqref{ADE} is globally well-posed in any space $W^{k,1}(\mathbb{R}^N)$ for $k \in \N$, see \cite{LR10}, using Sobolev embeddings, it follows that the solutions $u$ to \eqref{ADE} belong to $C([0,\infty),W^{k,p}(\mathbb{R}^N))$ for all $k \ge 0$ and $p\in [1,\infty]$. Also, $u(\cdot,t)$ remains non-negative and radially symmetric for all $t \ge 0$, and moreover we also have the mass conservation property, see \eqref{M} below. For more details on the well-posedness and regularity issues for \eqref{ADE}, we refer to \cite{KS11,LR10}.
\bigskip
\\ \indent
In the limit case $\eps=0$, the solution to \eqref{ADE} blows up after a finite time, provided the initial condition is sufficiently concentrated in a neighbourhood of the origin \cite{BKL09}. For more results about blow-up depending on the choice of the kernel $K$, see \cite{BRB11},\cite{BW99},\cite{CDFLS11},\cite{CJLV16},\cite{KS11},\cite{LS18}. For a more comprehensive review of the results and open problems, see the very recent book of the first author \cite{B-DeG}, especially Chapter 5, Section 4.
\\ \indent
In our work, we are concerned with the behaviour of solutions to \eqref{ADE} for $0 < \eps \ll 1$. In our previous paper \cite{BBKL20}, we obtained optimal estimates for Lebesgue norms of $u$. Heuristically, after the solution is allowed enough time to concentrate in a neighbourhood of zero, the behaviour of the (time-averaged) Lebesgue norms of $u$ is given by $\Vert u \Vert_p \sim \eps^{-N(1-1/p)}$.
\\ \indent
More rigorously, we proved that there are constants $\eps_*,\ T_*>0$, which only depend on the solution through the total mass and an additional parameter, such that for $0<\eps \le \eps_*$,
\begin{equation} \label{Lpstat}
\int_{0}^{T_*} \Big( \int_{\R^N} {u^p(x,t)\ dx} \Big)^{1/p}\ dt \sim \eps^{-N(1-1/p)},\ 1 \le p < \infty;
\end{equation}
moreover, this result remains true if we integrate over $x$ in a ball of radius $C \eps$ instead of the whole space. For the precise formulation, see Theorem~2.3, Corollary~2.4 and Lemma~4.1 in \cite{BBKL20}.
\\ \indent
To understand better small-scale behaviour of the solutions, it is relevant to look at norms beyond the Lebesgue setting. The Sobolev norms - which are natural candidates - have attracted much attention in models with physical motivation. 
Namely, in the pioneering works of Kuksin \cite{Kuk97} and \cite{Kuk99}, upper and lower estimates of these norms for solutions of the nonlinear Schr{\"o}dinger equation (with or without a random term) in a small dispersion regime have been obtained. After these seminal papers, study of the Sobolev norms in dispersive equations has become a very important field (see for example the paper \cite{CKSTT10} and the references therein). 
\\ \indent
Denoting by $\left\langle \cdot \right\rangle$ a time-average, dimensional analysis tells us that quantities of the type 
\begin{equation} \label{char}
\frac{\left\langle \Vert u \Vert_{\dot{H^m}} \right\rangle}{\left\langle \Vert u \Vert_{\dot{H}^{m+1}} \right\rangle},\qquad m \ge 0,
\end{equation}
(see Section~\ref{nota} for the notation) provide a \textit{characteristic length scale} of the solution. For a discussion, see the already mentioned papers of Kuksin \cite{Kuk97} and \cite{Kuk99},  as well as \cite[Chapter~6]{BoKu}. 
\\ \indent
The main results of our paper, Theorem~\ref{Hmupper} and Theorem~\ref{Hmlower}, state that, for the same $\eps_{*}$ and $T_{*}$ as in the statement of \eqref{Lpstat}, provided $0 \le \eps \le \eps_*$, we have
\begin{equation} \label{Hkstat}
\int_{0}^{T_*}  \Vert u(t) \Vert_{\dot{H^m}}\ dt \sim \eps^{-(2m+N)/2},\qquad m \in \N.
\end{equation}
\indent
Consequently, up to averaging in time, all the quantities given by \eqref{char} are of order $\varepsilon$, as is the radius of the balls on which at least an $\eps$-independent proportion of  mass is concentrated. To the best of our knowledge, our paper is the first one which studies systematically models from mathematical biology using all-order Sobolev norms.
\\ \indent
Moreover, our results for these norms - and therefore for the length scale - are sharp. Indeed, the upper and lower estimates only differ by a multiplicative constant which only depends on the initial condition through a finite number of parameters. This is a remarkable phenomenon, only previously observed in the Burgers equation and its generalisations. For more complex PDEs such as the 2D Navier--Stokes or the nonlinear Schr{\"o}dinger equation, such results are beyond the reach of today's mathematics.
\bigskip
\\ \indent
Our results are indeed similar to those obtained for the simpler Burgers equation and its fractional-dissipation and multidimensional analogues by the second author \cite{B14-ARMA,B14,B16ab,B18}. These papers were themselves inspired by the ideas and first results due to Biryuk \cite{Bir01}. Indeed, for Burgers-type equations the length scale is again the small parameter $\eps$, and we have sharp Sobolev norm estimates. More precisely, up to a rescaling factor corresponding to the dimension $N$, $-u$ has the same behaviour with respect to Lebesgue and Sobolev norms as the derivative of a Burgers solution.  In particular, the positivity of $u$ seems to play a role analogous to that of Oleinik's upper bound on the positive part of the gradient for a solution of the Burgers equation. Heuristically, it seems that the rescaling $N$-dependent factor in the power of $\eps$ is due to a difference of geometry of the singular zones in the limit $\eps \rightarrow 0$. Indeed, for \eqref{ADE} regions where the inviscid solution is not regular are of dimension zero (only the origin) and not shocks of codimension one as for the generalised Burgers equation. 
\\ \indent
Our methodology is essentially a combination of the approach used by the second author to study Sobolev norms in the papers cited above and of the arguments used by the three other authors to prove explosion under the concentration assumption in the paper \cite{BKL09} (see also \cite{KS10}). The most delicate issue is to estimate the contribution of the nonlinearity in the energy estimates, which requires a subtle analysis of the convolution term using the classical Hardy--Littlewood--Sobolev inequality along with the Gagliardo--Nirenberg inequality within the admissible ranges for the exponents.

\section{Notation, functional spaces and inequalities} \label{nota}

We denote by $M$ the total mass and recall that it is conserved by the flow of the equation \eqref{ADE}:
\begin{equation} \label{M}
\int_{\R^n}{u(x,t) dx}=M:=\int_{\R^n}{u_0(x) dx}.
\end{equation}
\\ \indent
For multiindices $\bi,\bj \in \Z_+^N$, provided $i_k \leq j_k,\ 1 \leq k \leq N$ (which we denote as $\bi \leq \bj$), we use the generalised binomial coefficient notation
$$
\binom{\bj}{\bi}=\prod_{k=1}^{N}{\binom{j_k}{i_k}}.
$$  
\\ \indent
For $N=1$ and a positive integer $k$, $u^{(k)}$ denotes the $k$-th spatial derivative of $u$,  while we use the notation $\partial_{\bi} u := \partial_{x_1}^{i_1}\ldots \partial_{x_N}^{i_N} u$ when $N > 1$ and $\bi=(i_k)_{1\le k\le N}$ is a multiindex.
\\ \indent
For $m \ge 0$ and $p\in [1,\infty]$, we will consider Lebesgue spaces $L^p(\mathbb{R}^N)$ and Sobolev spaces $W^{m,p}(\mathbb{R}^N)$. The Lebesgue norms will be denoted $\Vert \cdot \Vert_p$.  As usual, we set $H^m(\mathbb{R}^N)=W^{m,2}(\mathbb{R}^N)$, $m \in \N$. For $m\in \N$ and $p\in [1,\infty]$, we denote the homogeneous seminorm in $W^{m,p}(\mathbb{R}^N)$ by 
$$
\| u \|_{\dot{W}^{m,p}} := \sum_{|\bi|=m}{\| \partial_{\bi} u} \|_{p} \;\;\text{ with }\;\; \|\cdot \|_{\dot{H}^m} = \|\cdot\|_{\dot{W}^{m,2}}.
$$
\\ \indent
Throughout the paper, the notation $C$ and $C_i$, $i\ge 1$, is used for various positive numbers which may vary from line to line. These numbers depend only on the dimension $N$, and on the initial condition $u_0$ through the total mass $M$ and the quantity $\Lambda$ (see the beginning of the proof of Theorem~\ref{Hmlower}). The dependence upon additional parameters will be indicated explicitly. 
\bigskip
\\ \indent
Now we recall two classical inequalities.

\begin{lem} \label{GN} (\textit{The Gagliardo--Nirenberg Inequality,  \cite{BM19}}) 
For a $C^{\infty}$-smooth 
\\
function $v$ on $\R^N$, we have 
$$
\|v\|_{\dot{W}^{\beta,r}} \leq C \|v\|^{\theta}_{\dot{W}^{m,p}} \|v\|^{1-\theta}_{q},
$$
where $m>\beta\geq 0$, and $r$ is defined by
$$
\frac{N}{r}=\beta-\theta \Big( m-\frac{N}{p} \Big)+(1-\theta)\frac{N}{q},
$$
under the assumption $\beta/m \leq \theta < 1$ and with the exception of the case when $\beta=0$, $r=q=\infty$ and $m-N/p$ is a nonnegative integer.
\\
The constant $C$ depends also on $m,p,q,\beta,N$.
\end{lem}

\begin{lem} \label{HLS} (\textit{The Hardy--Littlewood--Sobolev Inequality.}) 
\\
\cite[Theorem~4.3]{Lilo01};\cite[V.1.3.]{EMSt}
\\
For a $C^{\infty}$-smooth function $v$ on $\R^N$, provided
$$
1<p,q<\infty,\ 1/p+\lambda/N=1/q+1,\ 0 < \lambda < N,
$$
we have 
$$
\Big\|\ |x|^{-\lambda} \ast v \Big\|_{q} \leq C \Vert v \Vert_{p}.
$$
where $\ast$ denotes the convolution. The constant $C$ depends on $p,\lambda,N$.
\end{lem}

\section{Upper estimates}
 
The results proved in this section still hold without the radial symmetry assumption on the initial condition, and also without the concentration assumption \eqref{A3}. Nevertheless, in that case we do not have corresponding lower estimates with the same power of the parameter $\eps$ proved in the next section.
\\ \indent
The scheme of the proof is very similar to that of the particular case $N=1,\ m=1$ already treated in \cite{BBKL20}.

\begin{thm} \label{Hmupper}
For $m \in \N$ and $t \geq 0$, we have
$$
\| u(t) \|_{\dot{H}^m} \leq \max\left\{ \| u_0 \|_{\dot{H}^m},\ C(m) M^{(N+2m+2)/2} \eps^{-(N+2m)/2} \right\}.
$$
\end{thm}

\begin{proof}
The case $m=0$ is dealt with in \cite{BBKL20}, to which we refer, see \cite[Lemma~4.1]{BBKL20} for $p=2$. From now on, we assume that $m \ge 1$.
\\
{\bf The case $N=1$.} Integrating by parts and using that for any $p \in [1,\infty]$, $(K' \ast v)_x=-2v$ for $v \in L^p(\mathbb{R}^N)$, we obtain
\begin{align*}
\frac{1}{2} \frac{\mathrm{d}}{\mathrm{d}t} \Vert u \Vert^2_{\dot{H}^m} 
\\ \nonumber
=& -\eps \Vert u \Vert^2_{\dot{H}^{m+1}}
-\int_{\R}{u^{(m)}(u\ (K' \ast u))^{(m+1)} dx} 
\\ \nonumber
= & -\eps \Vert u \Vert^2_{\dot{H}^{m+1}} - \int_{\R}{u^{(m)} u^{(m+1)} (K' \ast u) dx} 
\\ \nonumber
&- 
\sum_{k=0}^{m} \int_{\R}{\binom{m+1}{k} u^{(m)} u^{(k)} (K' \ast u^{(m-k)})_x dx}
\\ \nonumber
=&-\eps \Vert u \Vert^2_{\dot{H}^{m+1}}+\frac{1}{2} \int_{\R}{(u^{(m)})^2 (K' \ast u)_x dx}
\\ \nonumber
&- \sum_{k=0}^{m}  \int_{\R}{\binom{m+1}{k} u^{(m)} u^{(k)} (K' \ast u^{(m-k)})_x dx}
\\ \nonumber
=&-\eps \Vert u \Vert^2_{\dot{H}^{m+1}}-
\underbrace{\int_{\R}{(u^{(m)})^2\ u\ dx}}_{A_{m}} 
\\ \nonumber
&+ \sum_{k=0}^{m}  \underbrace{\int_{\R}{2 \binom{m+1}{k} u^{(m)} u^{(k)} u^{(m-k)} dx}}_{B_{km}}.
\end{align*}
We first get, using the H{\"o}lder and then the Gagliardo--Nirenberg inequalities, as well as \eqref{M},
\begin{align*}
|A_{m}| & \leq \Vert u \Vert_{\dot{W}^{m,\infty}}^2 \Vert u \Vert_1 
\leq C(m) (\Vert u \Vert_1^{1/(2m+3)} \Vert u \Vert_{\dot{H}^{m+1}}^{(2m+2)/(2m+3)})^2 \Vert u \Vert_1
\\ 
& = C(m) M^{(2m+5)/(2m+3)} \Vert u \Vert^{(4m+4)/(2m+3)}_{\dot{H}^{m+1}}.
\end{align*}
Similarly, we obtain
\begin{align*}
|B_{km}| \leq & C(k,m) \Vert u \Vert_{\dot{W}^{m,\infty}} \Vert u \Vert_{\dot{H}^k} \Vert u \Vert_{\dot{H}^{m-k}}
\\ \indent
 \leq & C(k,m) \left( \Vert u \Vert_1^{1/(2m+3)} \Vert u \Vert_{\dot{H}^{m+1}}^{(2m+2)/(2m+3)} \right) 
\\ \indent
& \qquad\qquad \times \left( \Vert u \Vert_1^{(2m+2-2k)/(2m+3)} \Vert u \Vert_{\dot{H}^{m+1}}^{(2k+1)/(2m+3)} \right) 
\\ \indent
& \qquad\qquad \times \left( \Vert u \Vert_1^{(2k+2)/(2m+3)} \Vert u \Vert_{\dot{H}^{m+1}}^{(2m-2k+1)/(2m+3)} \right)
\\ \indent
= & C(k,m) M^{(2m+5)/(2m+3)} \Vert u \Vert^{(4m+4)/(2m+3)}_{\dot{H}^{m+1}}.
\end{align*}
Consequently,
\begin{align} \label{decr}
& \frac{1}{2} \frac{\mathrm{d}}{\mathrm{d}t} \|u\|_{\dot{H}^m}^2 \leq - \eps \|u\|_{\dot{H}^{m+1}}^2 
+ C(m) M^{(2m+5)/(2m+3)} \Vert u \Vert^{(4m+4)/(2m+3)}_{\dot{H}^{m+1}}.
\end{align}
Now we observe that, interpolating $\|u(t)\|_{\dot{H}^m}$ between $\|u(t)\|_{\dot{H}^{m+1}}$ and $\|u(t)\|_{1}$ using the Gagliardo--Nirenberg inequality, we get, thanks to \eqref{M},
\begin{equation} \label{PhL200}
\|u(t)\|^{2/(2m+3)}_{\dot{H}^{m+1}} \ge C_{GN}(m) M^{-4/(2m+1)(2m+3)} \|u(t)\|^{2/(2m+1)}_{\dot{H}^{m}}.
\end{equation}
Our goal is now to show that the inequality \eqref{decr} implies that, for all $t \ge 0$,
\begin{align} \label{PhL202}
&\|u(t)\|_{\dot{H}^m} \leq U_m 
\\ \nonumber
&\equiv \max\left\{ \|u_0\|_{\dot{H}^m},\ C_{GN}(m)^{-(2m+1)/2} C(m)^{(2m+1)/2} M^{(2m+3)/2} \eps^{-(2m+1)/2} \right\}, 
\end{align}
with $C(m)$ is the same as in \eqref{decr} and $C_{GN}(m)$ the same as in \eqref{PhL200}. Indeed, for $\delta>0$, consider the set 
$$
A_{\delta} :=\left\{ t \geq 0:\, \|u(t)\|_{\dot{H}^m} \leq U_m+\delta \right\}.
$$
Clearly, $0 \in A_{\delta}$ and the time continuity of $u$ in $H^m(\mathbb{R}^N)$ ensures that 
$$
\tau_\delta := \sup\{ t\ge 0:\, [0,t]\subset A_\delta \}\in (0,\infty].
$$
Assume now for contradiction that $\tau_\delta<\infty$. The definition of $\tau_\delta$ implies that 
\begin{equation*}
\|u(\tau_\delta)\|_{\dot{H}^m}^2 = (U_m + \delta)^2 \ge \|u(t)\|_{\dot{H}^m}^2 \ \ \ {\rm for\ all\ \ }t\in (0,\tau_\delta). 
\end{equation*}
Hence,
\begin{equation}
\frac{\rm d}{\dt} \|u(\tau_\delta)\|_{\dot{H}^m}^2 \ge 0. \label{PhL201}
\end{equation}
We next infer from \eqref{decr}, \eqref{PhL200} and the definition of $U_m$ that 
\begin{align*}
& \frac{1}{2} \frac{\rm d}{\dt} \|u(\tau_\delta)\|_{\dot{H}^m}^2  
\\
\le & 
\eps \|u(\tau_\delta)\|_{\dot{H}^{m+1}}^{(4m+4)/(2m+3)} \left( - \|u(\tau_\delta)\|_{\dot{H}^{m+1}}^{2/(2m+3)} + C(m) M^{(2m+5)/(2m+3)} \eps^{-1} \right)
\\ 
\le & \eps \|u(\tau_\delta)\|_{\dot{H}^{m+1}}^{(4m+4)/(2m+3)} 
\big( - C_{GN}(m) M^{-4/(2m+1)(2m+3)} \|u(\tau_\delta)\|_{\dot{H}^{m}}^{2/(2m+1)} \big.
\\
& \big. + C(m) M^{(2m+5)/(2m+3)} \varepsilon^{-1} \big) < 0,
\end{align*}
which contradicts \eqref{PhL201}. Consequently, $\tau_\delta=\infty$ and $A_{\delta}=[0,\infty)$ for all $\delta>0$. Letting $\delta\to 0$ completes the proof of \eqref{PhL202}.
\\
\medskip
\indent
{\bf The case $N \geq 2$.} Let $\bi = (i_k)_{1\le k\le N} \in \mathbb{N}^N$ be a multiindex with $|\bi|=m$. By applying Leibniz' formula, it follows from \eqref{ADE} that $\partial_{\bi} u$ solves
\begin{equation*}
\left( \partial_{\bi} u \right)_t = \eps \Delta \left( \partial_{\bi} u \right) - \nabla\cdot \Big[ \sum_{0 \leq \bj \leq \bi} \binom{\bi}{\bj} \partial_{\bj} u\ (\nabla K* \partial_{\bi-\bj} u) \Big].
\end{equation*}
Multiplying the above equation by $\partial_{\bi} u$, integrating over $\mathbb{R}^N$, summing over all multiindices $\bi$ of length $m$ and then integrating by parts, we get 
\begin{align*}
\frac{1}{2} \frac{\mathrm{d}}{\mathrm{d}t} \|u\|_{\dot{H}^m}^2 
\\
= &- \eps \sum_{|\bi|=m} \| \nabla (\partial_{\bi} u) \|_{2}^2 
\\
&+ \sum_{|\bi|=m} \sum_{0 \leq \bj \leq \bi} \binom{\bi}{\bj} \int_{\mathbb{R}^N} \partial_{\bj} u \left( \nabla K* \partial_{\bi-\bj}  u \right) \cdot \nabla \partial_{\bi} u \dx
\\
= &- \eps \|u\|_{\dot{H}^{m+1}}^2 + \sum_{|\bi|=m} \int_{\mathbb{R}^N} \partial_{\bi} u \left( \nabla K* u \right) \cdot \nabla \partial_{\bi} u \dx
\\
& - \sum_{|\bi|=m} \sum_{0 \leq \bj < \bi} \binom{\bi}{\bj} \int_{\mathbb{R}^N} \nabla \partial_{\bj} u \cdot \left( \nabla K* \partial_{\bi-\bj} u \right) \partial_{\bi} u \dx
\\
& - \sum_{|\bi|=m} \sum_{0 \leq \bj < \bi} \binom{\bi}{\bj} \int_{\mathbb{R}^N} \partial_{\bj} u \left( \Delta K* \partial_{\bi-\bj} u \right) \partial_{\bi} u \dx,
\end{align*}
using that $\mathrm{div}(\nabla K * \partial_{\bi-\bj}u) = \Delta K * \partial_{\bi-\bj}u$. Since we can write
\begin{align*}
& \sum_{0\le\bj < \bi} \binom{\bi}{\bj} \int_{\mathbb{R}^N} \nabla \partial_{\bj} u \cdot \left( \nabla K* \partial_{\bi-\bj} u \right) \partial_{\bi} u \dx \\
= & \sum_{0\le\bj < \bi} \binom{\bi}{\bj} \sum_{k=1}^N \int_{\mathbb{R}^N} \partial_{x_k} \partial_{\bj} u \left( \partial_{x_k} K* \partial_{\bi-\bj} u \right) \partial_{\bi} u \dx \\
= & \sum_{|\bl|\le m} \sum_{r=1}^N \sum_{s=1}^N C_1(\bi,\bl,r,s) \int_{\mathbb{R}^N} \partial_{\bl} u \left( \partial_{x_r} \partial_{x_s} K* \partial_{\bi-\bl} u \right) \partial_{\bi} u \dx
\end{align*}
for some constants $C_1(\bi,\bl,r,s)\in \mathbb{R}$ and 
\begin{align*}
& \sum_{0 \leq \bj < \bi} \binom{\bi}{\bj} \int_{\mathbb{R}^N} \partial_{\bj} u \left( \Delta K* \partial_{\bi-\bj} u \right) \partial_{\bi} u \dx \\
= & \sum_{|\bl| \le m} \sum_{r=1}^N  C_2(\bi,\bj,r) \int_{\mathbb{R}^N} \partial_{\bl} u \left( \partial_{x_r}^2 K* \partial_{\bi-\bl} u \right) \partial_{\bi} u \dx
\end{align*}
for some constants $C_2(\bi,\bj,r)\in \mathbb{R}$, we obtain, after another integration by parts,
\begin{equation}
\begin{split}
\frac{1}{2} \frac{\mathrm{d}}{\mathrm{d}t} \|u\|_{\dot{H}^m}^2 = &- \eps \|u\|_{\dot{H}^{m+1}}^2 - \frac{1}{2} \sum_{|\bi|=m} \int_{\mathbb{R}^N} \left( \partial_{\bi} u\right)^2 \left( \Delta K* u \right) \dx \\
& + \sum_{|\bi|=m} \sum_{|\bl|\le m} \int_{\mathbb{R}^N} \partial_{\bl} u \left( P(\bi,\bl)K * \partial_{\bi-\bl} u \right) \partial_{\bi} u \dx,
\end{split} \label{sph1}
\end{equation}
where $P(\bi,\bl)$ are constant-coefficient differential operators of second order. We now split the last term of \eqref{sph1} to find, after moving partial derivatives inside the convolutions in an appropriate way, 
\begin{align*}
& \sum_{|\bi|=m} \sum_{|\bl|\le m} \int_{\mathbb{R}^N} \partial_{\bl} u \left( P(\bi,\bl)K * \partial_{\bi-\bl} u \right) \partial_{\bi} u \dx \\
= & \sum_{|\bi|=m}\quad \sum_{m-N+2\le |\bl|\le m} \int_{\mathbb{R}^N} \partial_{\bl} u \left( \partial_{\bi-\bl} P(\bi,\bl)K * u \right) \partial_{\bi} u \dx \\
& + \sum_{|\bi|=m}\quad \sum_{|\bl|\le m-N+1} \int_{\mathbb{R}^N} \partial_{\bl} u \left( P(\bi,\bl)K * \partial_{\bi-\bl} u \right) \partial_{\bi} u \dx.
\end{align*}
For the second term on the right hand side of the above identity, we observe that the conditions $|\bi|=m$ and $|\bl|\le m-N+1$ guarantee that $|\bi-\bl|\ge N-1$ and we can move $N-2$ partial derivatives from $\partial_{\bi-\bl} u$ on $P(\bi,\bl)K$ in the convolution $P(\bi,\bl)K * \partial_{\bi-\bl} u$ to find
\begin{align*}
& \sum_{|\bi|=m} \sum_{|\bl|\le m} \int_{\mathbb{R}^N} \partial_{\bl} u \left( P(\bi,\bl)K * \partial_{\bi-\bl} \right) \partial_{\bi} u \dx \\
= & \sum_{|\bi|=m}\quad \sum_{m-N+2\le |\bl|\le m} \int_{\mathbb{R}^N} \partial_{\bl} u \left( \partial_{\bi-\bl} P(\bi,\bl)K * u \right) \partial_{\bi} u \dx \\
& + \sum_{|\bi|=m}\quad \sum_{|\bl|\le m-N+1}\quad \sum_{|\bj|=m-|\bl|-N+2} \int_{\mathbb{R}^N} \partial_{\bl} u \left( Q(\bi,\bl,\bj)K * \partial_{\bj} u \right) \partial_{\bi} u \dx,
\end{align*}
where $Q(\bi,\bl,\bj)$ are constant-coefficient differential operators of order $N$. Inserting the above identity in \eqref{sph1} and computing the partial derivatives of $K$ leads to

\begin{align*}
& \frac{1}{2} \frac{\mathrm{d}}{\mathrm{d}t} \|u\|_{\dot{H}^m}^2 \leq - \eps \|u\|_{\dot{H}^{m+1}}^2 
+ \frac{ N-1}{2} \sum_{|\bi|=m} \underbrace{\int_{\mathbb{R}^N} (\partial_{\bi} u)^2 \left( |x|^{-1}* u\right)  \dx}_{A_\bi}
\\
& +  \sum_{|\bi|=m}\quad \sum_{ m-N+2 \le |\bl| \le m}\  C(\bi,\bl) \underbrace{\int_{\mathbb{R}^N} 
|\partial_{\bl} u| \left( 
|x|^{-(m- |\bl|+1)}*  u \right) |\partial_{\bi} u| \dx}_{D_{\bi,\bl}}
\\
& + \sum_{|\bi|=m}\ \sum_{|\bl|\le m-N+1} C(\bi,\bl) \underbrace{ \sum_{|\bj|=m-|\bl|-N+2} \int_{\mathbb{R}^N} 
|\partial_{\bl} u| \left( |x|^{-(N-1)} * |\partial_{\bj} u| \right) |\partial_{\bi} u| \dx}_{E_{\bi,\bl}}.
\end{align*}
Now it remains to estimate all the terms using first the H{\"o}lder, and then the Gagliardo--Nirenberg and the Hardy--Littlewood--Sobolev inequalities, along with the mass conservation \eqref{M}. First,
\begin{align*}
|A_\bi| \leq  &\Vert (\partial_{\bi} u)^2 \Vert_{2N/(2N-1)} \left\Vert |x|^{-1}* u \right\Vert_{2N} \leq C(m) \Vert \partial_{\bi} u\Vert^2_{4N/(2N-1)} \Vert  u \Vert_{2N/(2N-1)}
\\
\leq & C(m) \Vert u \Vert^2_{\dot{W}^{m,4N/(2N-1)}} \Vert u \Vert_{2N/(2N-1)}
\\
\leq & C(m) \Big(\Vert u\Vert_1^{3/(2m+N+2)} \Vert u \Vert^{(4m+2N+1)/(2m+N+2)}_{\dot{H}^{m+1}}\Big) 
\\
& \times \Big(\Vert u \Vert_1^{(2m+N+1)/(2m+N+2)} \Vert u \Vert^{1/(2m+N+2)}_{\dot{H}^{m+1}}\Big) 
\\
\leq & C(m) M^{(2m+N+4)/(2m+N+2)} \Vert u \Vert^{(4m+2N+2)/(2m+N+2)}_{\dot{H}^{m+1}}.
\end{align*}
Next, when $\bl$ satisfies $0\le m-|\bl| \le N-2$,
\begin{align*}
|D_{\bi, \bl}|  \leq &\  \Vert \partial_{\bl} u \Vert_{2N/(N-m+|\bl|)}
\left\Vert |x|^{-(m-|\bl|+1)} *  u \right\Vert_{4N/(2m-2|\bl|+1)} \Vert \partial_{\bi} u \Vert_{4N/(2N-1)}
\\
\leq &\ C(\bi,\bl) \Vert u \Vert_{\dot{W}^{|\bl|,\ 2N/(N-m+L)}} \Vert u \Vert_{4N/(4N-2(m-|\bl|)-3)} \Vert u \Vert_{\dot{W}^{m,4N/(2N-1)}}
\\
\leq &\ C(\bi,\bl) (\Vert u \Vert_{1}^{1-\alpha} \Vert u \Vert_{\dot{H}^{m+1}}^{\alpha} )
									 (\Vert u \Vert_{1}^{1-\beta} \Vert u \Vert_{\dot{H}^{m+1}}^{\beta} )
									 (\Vert u \Vert_{1}^{1-\gamma} \Vert u \Vert_{\dot{H}^{m+1}}^{\gamma} )
\\
\leq &\ C(\bi,\bl) M^{(2m+N+4)/(2m+N+2)} \Vert u \Vert^{(4m+2N+2)/(2m+N+2)}_{\dot{H}^{m+1}},
\end{align*}
where
$$
\alpha={\frac{m+|\bl|+N}{2m+N+2}},\quad \beta=\frac{2(m-|\bl|)+3}{2(2m+N+2)},\quad \gamma={ \frac{4m+2N+1}{2(2m+N+2)}.}
$$
Finally, when $(\bl,\bj)$ satisfies $m-|\bl|\ge N-1$ and $|\bj|=m-|\bl|-N+2$, 
\begin{align*}
|E_{\bi,\bl}|  \leq &\ \Vert \partial_{\bl} u \Vert_{(4m+4)/(2m-1)}
\left\Vert |x|^{-(N-1)} * |\partial_{\bj} u| \right\Vert_{4m+4} \Vert \partial_{\bi} u \Vert_{(2m+2)/(m+2)}
\\
\leq &\ C(\bi,\bl) \Vert u \Vert_{\dot{W}^{|\bl|,(4m+4)/(2m-1)}} \Vert u \Vert_{\dot{W}^{m-|\bl|-N+2,\ N(4m+4)/(N+4m+4)}} 
\\ 
& \times \Vert u \Vert_{\dot{W}^{m,(2m+2)/(m+2)}}
\\
\leq &\ C(\bi,\bl) (\Vert u \Vert_{1}^{1-\delta} \Vert u \Vert_{\dot{H}^{m+1}}^{\delta} )
									 (\Vert u \Vert_{1}^{(1-\delta')} \Vert u \Vert_{\dot{H}^{m+1}}^{\delta'} )
\\
& \times(\Vert u \Vert_{1}^{1/(m+1)} \Vert u \Vert_{\dot{H}^{m+1}}^{m/(m+1)} )
\\
\leq &\ C(\bi,\bl) M^{(2m+N+4)/(2m+N+2)} \Vert u \Vert^{(4m+2N+2)/(2m+N+2)}_{\dot{H}^{m+1}},
\end{align*}
where
$$
\delta=\frac{2(m+1)(2L+N)+3N}{2(2m+N+2)(m+1)},\quad \delta'=\frac{4(m+1)(m-L+1)-N}{2(2m+N+2)(m+1)}.
$$
Summing up all the above estimates, we get 
\begin{align*}
&\frac{1}{2} \frac{\mathrm{d}}{\mathrm{d}t} \|u\|_{\dot{H}^m}^2 \\
\leq &  - \eps \|u\|_{\dot{H}^{m+1}}^2 + C(m) M^{(2m+N+4)/(2m+N+2)} \Vert u \Vert^{(4m+2N+2)/(2m+N+2)}_{\dot{H}^{m+1}} \\
=&   \|u\|_{\dot{H}^{m+1}}^{(4m+2N+2)/(2m+N+2)} \\
& \times \left( C(m) M^{(2m+N+4)/(2m+N+2)} - \varepsilon \|u\|_{\dot{H}^{m+1}}^{2/(2m+N+2)} \right). 
\end{align*}
From this energy inequality combined with the following consequence of
\\
the Gagliardo--Nirenberg inequality
\begin{equation*}
\|u\|_{\dot{H}^{m+1}}^{2/(2m+N+2)} \ge C(m)\ M^{-4/(2m+N)(2m+N+2)} \|u\|_{\dot{H}^m}^{2/(2m+N)}
\end{equation*}
we deduce, arguing as in the case $N=1$, that
\begin{equation*}
\|u(t)\|_{\dot{H}^m} \le \max\left\{ \|u_0\|_{\dot{H}^m} , C(m) M^{(N+2m+2)/2} \eps^{-(N+2m)/2} \right\}, \qquad t\ge 0,
\end{equation*}
as announced.
\end{proof}

\begin{rmk}
After any given time $\tau>0$, the estimates above will hold with a $\tau$-dependent upper bound as a consequence of an interplay between the smoothing properties of the heat kernel and a genuinely nonlinear effect. Moreover, this upper bound will only depend on $u_0$ through the single quantity $M$. To obtain this result, one uses the same method as for the $u_0$-uniform upper estimates for the $H^m$-norms given by \cite[Lemma~53]{B14-ARMA} for solutions of the Burgers equation.
\end{rmk}

%

\section{Lower estimates for Sobolev norms}

Here --- unlike in the previous section --- the positivity and radial symmetry assumptions \eqref{A2} as well as the concentration assumption \eqref{A3} on the initial condition $u_0$ play a crucial role.

\begin{thm} \label{Hmlower}
Let $m\in \N$. For some explicit numbers  $\eps_\ast>0$,  $T_\ast>0$ and $C_\ast(m)>0$, independent of $\eps$, 
the following inequality holds true:
\be
\int_0^{T_\ast} \Vert u \Vert_{\dot{H}^{m}} \dt\ge C_\ast(m) \eps^{-(2m+N)/2},
\qquad 
\text{for all} \quad  \eps\in(0,\eps_\ast). 
\label{concen}
\ee
\end{thm}

\begin{proof}
For $m=0$, see \cite[Corollary~2.4]{BBKL20} for $p=2$. Indeed, if \eqref{u0conc} is true, then there exists $\Lambda>0$ such that the assumption on the initial condition \cite[Eq.~(2.8)]{BBKL20} holds true; see also \cite[Remark 2.7]{BBKL20}.
\\
For $m\ge 1$, we infer from the Gagliardo--Nirenberg inequality that
\begin{align*}
\|u(t)\|_2 \le C_{GN}(m) \|u(t)\|_{\dot{H}^m}^{N/(N+2m)} \|u(t)\|_1^{2m/(N+2m)}.
\end{align*}
Hence, by \eqref{M},
\begin{equation*}
\|u(t)\|_{\dot{H}^m} \ge C_{GN}(m)^{-(N+2m)/N} M^{-2m/N}\ \|u(t)\|_2^{(N+2m)/N}, \qquad t\ge 0,
\end{equation*}
and it follows from H\"older's inequality and the already established lower bound \eqref{concen} for $m=0$ that
\begin{align*}
& \left( \int_0^{T_{*}} \|u(t)\|_{\dot{H}^m}^2 \dt \right)^{1/2} \ge C(m) M^{-2m/N}\ \left( \int_0^{T_{*}} \|u(t)\|_2^{2(N+2m)/N} \dt \right)^{1/2} 
\\
& \ge C(m) M^{-2m/N}\ T_{*}^{-(N+4m)/2N} \left( \int_0^{T_{*}} \|u(t)\|_2 \dt \right)^{(N+2m)/N} \\
 & \ge C_{*}^{(N+2m)/N} C(m) M^{-2m/N}\ T_{*}^{ -(N+4m)/2N} \eps^{-(N+2m)/2},
\end{align*}
which completes the proof.
\end{proof}

\section*{Acknowledgements}
This work was partially supported by the French-Polish PHC Polonium grant 40592NJ (for all four authors), and NCN 2016/23/B/ST1/00434 (for the first author). 
\\

\bibliographystyle{plain}
\bibliography{BBKL-biblio}

\begin{thebibliography}{10}

\bibitem{BRB11}
J.~Bedrossian, N.~Rodr\'{\i}guez, and A.L. Bertozzi.
\newblock Local and global well-posedness for aggregation equations and
  {P}atlak-{K}eller-{S}egel models with degenerate diffusion.
\newblock {\em Nonlinearity}, 24(6):1683--1714, 2011.

\bibitem{B-DeG}
P.~Biler.
\newblock {\em {Singularities of solutions to chemotaxis systems.}}, volume 6,
  De Gruyter Series in Mathematics and Life Sciences.
\newblock Berlin: De Gruyter, 2020.

\bibitem{BBKL20}
P.~Biler, A.~Boritchev, G.~Karch, and P.~Lauren{\c c}ot.
\newblock Concentration phenomena in a diffusive aggregation model.
\newblock {\em Preprint. arXiv:2001.06218}.

\bibitem{BKL09}
P.~Biler, G.~Karch, and Ph. Lauren\c{c}ot.
\newblock Blowup of solutions to a diffusive aggregation model.
\newblock {\em Nonlinearity}, 22(7):1559--1568, 2009.

\bibitem{BW99}
P.~{Biler} and W.A. {Woyczynski}.
\newblock {Global and exploding solutions for nonlocal quadratic evolution
  problems}.
\newblock {\em {SIAM J. Appl. Math.}}, 59(3):845--869, 1999.

\bibitem{Bir01}
A.~Biryuk.
\newblock {Spectral properties of solutions of the Burgers equation with small
  dissipation}.
\newblock {\em Functional Analysis and its Applications}, 35:1:1--12, 2001.

\bibitem{B14-ARMA}
A.~{Boritchev}.
\newblock {Decaying turbulence in the generalised Burgers equation}.
\newblock {\em {Arch. Ration. Mech. Anal.}}, 214(1):331--357, 2014.

\bibitem{B14}
A.~Boritchev.
\newblock Turbulence in the generalised {B}urgers equation.
\newblock {\em Uspekhi Mat. Nauk}, 69:6(6(420)):3--44, 2014.

\bibitem{B16ab}
A.~Boritchev.
\newblock Multidimensional potential {B}urgers turbulence, and {E}rratum.
\newblock {\em Comm. Math. Phys.}, 342, 346(2):441--489, 369--370, 2016.

\bibitem{B18}
A.~Boritchev.
\newblock Decaying turbulence for the fractional subcritical {B}urgers
  equation.
\newblock {\em Discrete Contin. Dyn. Syst.}, 38(5):2229--2249, 2018.

\bibitem{BoKu}
A.~Boritchev and S.~Kuksin.
\newblock {\em One-dimensional turbulence and the stochastic Burgers equation}.
\newblock Submitted.

\bibitem{BM19}
H.~Brezis and P.~Mironescu.
\newblock Where {S}obolev interacts with {G}agliardo--{N}irenberg.
\newblock {\em Journal of Functional Analysis}, 277:2839--2864, 2019.

\bibitem{CDFLS11}
J.~A. Carrillo, M.~DiFrancesco, A.~Figalli, T.~Laurent, and D.~Slep\v{c}ev.
\newblock Global-in-time weak measure solutions and finite-time aggregation for
  nonlocal interaction equations.
\newblock {\em Duke Math. J.}, 156(2):229--271, 2011.

\bibitem{CJLV16}
J.A. Carrillo, F.~James, F.~Lagouti{\`e}re, and N.~Vauchelet.
\newblock {The Filippov characteristic flow for the aggregation equation with
  mildly singular potentials}.
\newblock {\em {Journal of Differential Equations}}, 260(1):304--338, 2016.
\newblock 33 pages.

\bibitem{CKSTT10}
J.~Colliander, M.~Keel, G.~Staffilani, H.~Takaoka, and T.~Tao.
\newblock Transfer of energy to high frequencies in the cubic defocusing
  nonlinear {S}chr\"{o}dinger equation.
\newblock {\em Invent. Math.}, 181(1):39--113, 2010.

\bibitem{HP09}
T.~{Hillen} and K.~J. {Painter}.
\newblock {A user's guide to PDE models for chemotaxis.}
\newblock {\em {J. Math. Biol.}}, 58(1-2):183--217, 2009.

\bibitem{KS10}
G.~Karch and K.~Suzuki.
\newblock Spikes and diffusion waves in a one-dimensional model of chemotaxis.
\newblock {\em Nonlinearity}, 23(12):3119--3137, 2010.

\bibitem{KS11}
G.~Karch and K.~Suzuki.
\newblock Blow-up versus global existence of solutions to aggregation
  equations.
\newblock {\em Appl. Math. (Warsaw)}, 38(3):243--258, 2011.

\bibitem{Kuk97}
S.~Kuksin.
\newblock {On turbulence in nonlinear Schr{\"o}dinger equations}.
\newblock {\em Geometric and Functional Analysis}, (7):783--822, 1997.

\bibitem{Kuk99}
S.~Kuksin.
\newblock {Spectral properties of solutions for nonlinear PDEs in the turbulent
  regime}.
\newblock {\em Geometric and Functional Analysis}, 9:141--184, 1999.

\bibitem{LS18}
L.~Lafleche and S.~Salem.
\newblock {Fractional {K}eller-{S}egel equation: global well-posedness and
  finite time blow-up}.
\newblock {\em Comm. Math. Sci.}, 17(8):2055--2087, 2019.

\bibitem{LR10}
D.~Li and J.L. Rodrigo.
\newblock Wellposedness and regularity of solutions of an aggregation equation.
\newblock {\em Rev. Mat. Iberoam.}, 26(1):261--294, 2010.

\bibitem{Lilo01}
E.H. Lieb and M.~Loss.
\newblock {\em Analysis}, volume~14 of {\em Graduate Studies in Mathematics}.
\newblock American Mathematical Society, Providence, RI, second edition, 2001.

\bibitem{EMSt}
E.M. {Stein}.
\newblock {\em {Singular integrals and differentiability properties of
  functions}}.
\newblock Princeton: Princeton University Press, 1970.

\end{thebibliography}

\end{document}